\documentclass[reqno,12pt]{amsart}

\usepackage[margin=1in]{geometry}
\usepackage{amsmath,amssymb,amsthm}
\usepackage[hidelinks]{hyperref}
\usepackage{microtype}

\newcommand{\R}{\mathbb R}
\newcommand{\ind}{\mathbf 1}
\newcommand{\relint}{\operatorname{relint}}
\DeclareMathOperator{\length}{length}

\newtheorem{theorem}{Theorem}
\newtheorem{lemma}{Lemma}

\newtheorem{corollary}{Corollary}
\theoremstyle{remark}
\newtheorem*{remark}{Remark}

\title{Polyhedral Subspaces of $L_{p}$ and Polars of Zonotopes}
\author[V.~Yaskin]{Vladyslav Yaskin}
\address{Department of Mathematical \& Statistical Sciences, University of Alberta, Edmonton, Alberta, T6G 2G1, Canada}
\email{yaskin@ualberta.ca}

\date{}

\begin{document}

\begin{abstract} Let $K$ be an origin-symmetric full-dimensional convex polytope in $\mathbb R^n$, $n\ge 3$. We show that, for $-n+3<p<1$, the normed space  $(\mathbb R^n,\|\cdot\|_K)$ embeds in $L_p$ if and only if it embeds in $L_1$. The case of $p\le 0$ is understood in the generalized sense. 
In particular, an origin-symmetric full-dimensional convex
polytope $K\subset\R^n$, $n\ge 5$, is   an intersection body if and only if  $K$ is the polar of a  zonotope.
\end{abstract}

\maketitle
 
\section{Introduction}

The problem of characterizing finite-dimensional normed spaces that embed isometrically into $L_p$ has a long history, going back to the work of P.~L\'evy.  
For $p>0$, let $\mathcal I_p$ denote the class of unit
balls of $n$-dimensional normed spaces  $(\mathbb R^n,\|\cdot\|)$ that embed isometrically in
$L_p$.  The notion of embedding into $L_p$ admits natural extensions to $-n<p<0$, introduced by Koldobsky \cite{Koldobsky1999}, and to the  case $p=0$, introduced by Kalton, Koldobsky, Yaskin, and Yaskina \cite{KaltonKoldobskyYaskinYaskina}. We use the same notation $\mathcal I_p$ for the corresponding classes. Precise definitions and the Fourier-analytic characterizations of these classes are given in Section~\ref{sec:definitions}.

The classical inclusion theorem of Bretagnolle, Dacunha-Castelle, and Krivine \cite{BDK}, together with its extensions to non-positive parameters \cite{Koldobsky1999}, \cite{KaltonKoldobskyYaskinYaskina}, implies that 
\begin{align}\label{I1inIp}
	\mathcal I_1\subset \mathcal I_p, \qquad -n<p<1. 
	\end{align}
In general, these inclusions are strict. In fact, for every
$-n<p<1$ there exist finite-dimensional normed spaces that embed
 into $L_p$ but not into $L_1$; see
\cite{Koldobsky1996}, \cite{Koldobsky1999}, \cite{KaltonKoldobsky}, \cite{KaltonKoldobskyYaskinYaskina}.
 
The purpose of this paper is to show that the situation changes
drastically for polyhedral normed spaces. Recall that  a polytope is the convex hull of finitely many points.  

\begin{theorem}\label{thm:main-intro}
	Let $K$ be an  origin-symmetric full-dimensional polytope in $\mathbb R^n$,
	$n\geq 3$.  If $K\in \mathcal I_p$, for  some
	\( p\in (
	-n+3,1)
	\),  then $K\in\mathcal I_1$.
\end{theorem}

\begin{remark}
	Every origin-symmetric convex body in $\mathbb R^n$ belongs to $\mathcal I_p$ for all $p\in (-n,-n+3]$; see  \cite[Corollary 4.9]{KoldobskyBook} for the case $p<0$ and \cite{KaltonKoldobskyYaskinYaskina} for the case $p=0$. On the other hand, the cube $B^n_\infty$ does not belong to $\mathcal I_1$ in dimensions $n\ge 3$ (see \cite{Bolker}). Hence,  the lower endpoint $p=-n+3$ in Theorem \ref{thm:main-intro} is sharp.
\end{remark}

Together with the reverse inclusion
\eqref{I1inIp},
 Theorem~\ref{thm:main-intro} gives a complete classification
of the polyhedral members of $\mathcal I_p$ for 	$ p\in (
-n+3,1)$:
\[
\mathcal I_p\cap\mathcal P_n
=
\mathcal I_1\cap\mathcal P_n,
\qquad -n+3<p<1,
\]
where $\mathcal P_n$ denotes the class of origin-symmetric
full-dimensional polytopes in $\mathbb R^n$.   

There is a natural convex-geometric interpretation of  this
statement.  Recall  that a {zonotope} is a finite Minkowski
sum of line segments, and a {zonoid} is a Hausdorff limit of
zonotopes.  A classical characterization of zonoids shows that a
finite-dimensional normed space embeds isometrically in $L_1$ if and
only if the polar of its unit ball is a zonoid; see, for example, \cite{Bolker}.  Moreover, a polytope
which is a zonoid is necessarily a zonotope; see \cite[Corollary~3.5.7]{Schneider}.  Consequently, for a
polyhedral normed space with unit ball $K$,
\[
K\in\mathcal I_1
\quad\Longleftrightarrow\quad
K^\circ\ \text{is a zonotope}.
\]
Thus Theorem~\ref{thm:main-intro} can be  stated in the following  geometric form:
\[
K\in\mathcal I_p\cap \mathcal P_n \qquad (-n+3<p\leq1)
\quad\Longleftrightarrow\quad
K^\circ\ \text{is a zonotope}.
\]

For negative integers, the classes $\mathcal I_{p}$ can be defined in the language of  intersection bodies. In this setting, Theorem \ref{thm:main-intro} yields the following. If $1\leq k<n-3$, then
\[
K\text{ is a polytopal $k$-intersection body}
\quad\Longleftrightarrow\quad
K^\circ\text{ is a zonotope}.
\]
In particular, for $n\geq5$, the polar of every polytopal intersection
body is a zonotope.

For a linear subspace
\(H\subset\mathbb R^n\), let \(\mathcal I_p(H)\) denote the
class of convex bodies in  \(H\) that are unit balls of finite-dimensional normed spaces that embed in $L_p$, $p>-\mathrm{dim}\, H$. If $K$ is an origin-symmetric convex body in $\mathbb R^n$ that belongs to $\mathcal I_p$ for $p>-n+1$, then $K\cap H$ belongs to $\mathcal I_p(H)$ for each $(n-1)$-dimensional subspace $H\subset \mathbb R^n$; see the discussion in Section \ref{sec:section-lifting}.  
 We show that in the polytopal setting the reverse implication also holds. Let $n\ge 4$ and $-n+4<p\leq1$.	If \(K\subset\mathbb R^n\)  is an origin-symmetric
full-dimensional polytope, such that 	\(
K\cap H\in\mathcal I_p(H)
\), 
for every $(n-1)$-dimensional subspace \(H\subset\mathbb R^n\), then \(K^\circ\) is a zonotope, and in particular, 
	\(
	K\in\mathcal I_p.
	\)

\section{Definitions and notation}\label{sec:definitions}

One of the main tools in this paper is the Fourier transform of distributions. We recall only the facts that will be needed below and refer the reader to \cite{GelfandShilov,KoldobskyBook} for further background.

Let $\mathcal S(\mathbb R^n)$ denote the Schwartz space of rapidly decreasing infinitely differentiable complex-valued functions on $\mathbb R^n$. The elements of this space are called test functions. By
$\mathcal S'(\mathbb R^n)$ we denote  the space of tempered distributions, i.e., continuous linear functionals on $\mathcal S(\mathbb R^n)$. For $\varphi\in \mathcal S(\mathbb R^n)$
its Fourier transform is defined by
\[
\widehat{\varphi}(\xi)
=
\int_{\mathbb R^n}
e^{-i\langle x,\xi\rangle}\varphi(x)\,dx.
\]
It is known that $\widehat\varphi$ is also a test function. 
The Fourier transform of a distribution $f\in\mathcal S'(\mathbb R^n)$ is a distribution $\widehat f$ defined by its action
on test functions as follows:
$$\langle \widehat f, \varphi\rangle =  \langle  f, \widehat \varphi\rangle, \qquad \mbox{for all } \varphi \in\mathcal S(\mathbb R^n).$$

For $g\in\R^n$, write
\[
\partial_g=\sum_{m=1}^n g_m\partial_m,
\]
where $\partial_m$ stands for the partial derivative with respect to the $m$th coordinate.

The derivative of a distribution $f$ is defined by
\[
\langle\partial_gf,\varphi\rangle
=-\langle f,\partial_g\varphi\rangle, \qquad \varphi \in\mathcal S(\mathbb R^n).
\]
For a test function $\varphi$  we have
\[\partial_g \widehat\varphi = - i \widehat{\langle \cdot, g \rangle \varphi}.
\]

We say that a distribution $f\in\mathcal S'(\mathbb R^n)$ is positive definite if its Fourier transform is a positive distribution, that is $$\langle \widehat f, \varphi\rangle\ge 0,$$
for every non-negative test function $\varphi$.

 We say that a distribution $f$ is positive outside the origin if
 \[
 \langle f,\varphi\rangle\geq0
 \]
 for every non-negative test function 
 $\varphi$ supported outside the origin.

For $p>0$, we say that a finite-dimensional normed space
$(\mathbb R^n,\|\cdot\|)$ {embeds isometrically in $L_p$} if
there is a linear isometry from this space into some $L_p$ space.
By  the classical P. L\'evy representation (see \cite[Lemma 6.4]{KoldobskyBook}), $(\mathbb R^n,\|\cdot\|)$ embeds isometrically in $L_p$, $p>0$, if and only if there is a finite even Borel measure $\mu$ on the unit sphere $S^{n-1}$
such that
\begin{align}\label{eq:p-integral-representation}
\|x\|^p
=
\int_{S^{n-1}}
|\langle x,\theta\rangle|^p\,d\mu(\theta),
\qquad x\in\mathbb R^n.
\end{align}

For $p>0$, that is not an even integer, Koldobsky \cite{Koldobsky1992} gave the following Fourier-analytic version of this characterization: a finite-dimensional normed space
$(\mathbb R^n,\|\cdot\|)$ embeds in $L_p$ if and only if  $ \Gamma\left(-\frac p2\right) \widehat{\|\cdot\|^p} $ is a positive distribution outside the origin.  In particular, if $0<p<1$, then $(\mathbb R^n,\|\cdot\|)$ embeds in $L_p$ if and only if  $ - \widehat{\|\cdot\|^p} $ is a positive distribution outside the origin.

The notion of embedding in $L_p$ with $-n<p<0$ was introduced by Koldobsky \cite{Koldobsky1999} as an analytic extension of the  P. L\'evy representation. He also proved the following Fourier characterization: a finite-dimensional normed space $(\mathbb R^n,\|\cdot\|)$ embeds in $L_p$ for $-n<p<0$ if and only if  
$\|x\|^p$ is a positive definite distribution.

The limiting case $p=0$ was introduced in \cite{KaltonKoldobskyYaskinYaskina}. A finite-dimensional normed space $(\mathbb R^n,\|\cdot\|)$ is said to embed in $L_0$ if there exist a probability measure $\mu$ on $S^{n-1}$ and a constant $C\in\mathbb R$ such that \begin{equation*}
	 \log\|x\| = \int_{S^{n-1}} \log|\langle x,\theta\rangle|\,d\mu(\theta)+C, \qquad x\neq0. 
\end{equation*} 
Equivalently, $(\mathbb R^n,\|\cdot\|)$ embeds in $L_0$ if  $-\widehat{\log\|\cdot\|}$  is a positive distribution outside the origin; see \cite{KaltonKoldobskyYaskinYaskina}.

Throughout the paper, $\mathcal I_p$ denotes the class of
origin-symmetric convex bodies in $\mathbb R^n$ that are unit balls
of finite-dimensional spaces embedding in $L_p$ in the corresponding
sense.

The classes  $\mathcal I_p$ are known to obey the following inclusions. 
For
\(
0<q<p\leq2,
\)
one has  
\(
\mathcal I_p\subset\mathcal I_q
\);  see \cite{BDK}. 
In particular,
\(
\mathcal I_1\subset\mathcal I_q\),
\( 0<q<1.
\) Koldobsky \cite{Koldobsky1999} extended this result to negative $q$: Every
$n$-dimensional subspace of $L_p$, $0 < p \le 2$, embeds in $L_q$ for every $-n < q < 0$. The limiting case of $L_0$ was treated in \cite{KaltonKoldobskyYaskinYaskina}.
On the other hand, there was a series of works with the aim of showing that these classes are different; see \cite{Koldobsky1996}, \cite{KaltonKoldobsky}, \cite{Milman2006}, \cite{Schlieper}, \cite{Yaskin2008}, \cite{KaltonZymonopoulou}.
In particular, combining the results obtained in \cite{Koldobsky1996}, \cite{Koldobsky1999}, \cite{KaltonKoldobsky}, 	\cite{KaltonKoldobskyYaskinYaskina}, we see that for every
$-n<p<1$ there exist finite-dimensional normed spaces that embed
 into $L_p$ but not into $L_1$.

We now turn to the geometric notions used in the paper. For further
background, see  \cite{Gardner}, \cite{KoldobskyBook}, \cite{Schneider}.
We say that $K\subset\mathbb R^n$ is a star body if it is compact, star-shaped with respect to the origin and its  radial function
\[
\rho_K(\theta)
=
\max\{t\geq0:t\theta\in K\},
\qquad \theta\in S^{n-1},
\]
is positive and continuous.  Its Minkowski functional is given by
\[
\|x\|_K
=
\inf\{\lambda>0:x\in\lambda K\}.
\]
It is easy to see that
\[
\rho_K(\theta)=\|\theta\|_K^{-1}.
\]
A convex compact set with nonempty interior is called a convex body. The Minkowski functional $\|\cdot\|_K$ is a norm precisely when $K$ is an origin-symmetric convex body. 

The radial metric on the class of star bodies is defined by \[ d(K,L) = \max_{\theta\in S^{n-1}} |\rho_K(\theta)-\rho_L(\theta)|. \]

Following Lutwak \cite{Lutwak}, an origin-symmetric star body $K$ is called the
{intersection body} of an origin-symmetric star body $L$ if
\[
\rho_K(\theta)
=
\operatorname{vol}_{n-1}(L\cap\theta^\perp),
\qquad \theta\in S^{n-1}.
\]
The  closure of the collection of such bodies in the radial metric is called the class
of {intersection bodies}. Intersection bodies played a key role in the solution of the celebrated Busemann-Petty problem; see  \cite[Chapter 1]{KoldobskyBook} for historical details.

 Koldobsky \cite{Koldobsky1999Israel} generalized the notion of  intersection bodies to higher codimensions. Let  $1\leq k<n$ be an integer and let $K$ and $L$ be origin-symmetric star bodies in $\mathbb R^n$. We say that $K$ is a
{$k$-intersection body of $L$} if
\[
\operatorname{vol}_k(K\cap H^\perp)
=
\operatorname{vol}_{n-k}(L\cap H)
\]
for every $(n-k)$-dimensional subspace $H\subset\mathbb R^n$.    The
class of $k$-intersection bodies is obtained by taking radial limits of
bodies arising in this way. The class of $1$-intersection bodies coincides with the class of  intersection bodies.

Koldobsky's Fourier characterization states that
an origin-symmetric star body $K$  is a $k$-intersection body if and only if $
\|\cdot\|_K^{-k}$
  is a positive definite distribution;
see \cite{KoldobskyFunctional,KoldobskyBook}.  Consequently, among
origin-symmetric convex bodies,
\[
K\text{ is a $k$-intersection body}
\quad\Longleftrightarrow\quad
K\in\mathcal I_{-k}.
\]

Throughout the paper, the term ``$k$-intersection body'' is used in this
sense. It should not be confused with the class of generalized
$k$-intersection bodies introduced by G.~Zhang \cite{ZhangSections}.

For sets $A,B\subset\mathbb R^n$, their Minkowski sum is
\[
A+B=\{x+y:x\in A,\ y\in B\}.
\] 
The support function of a convex body $K$ is
\[
h_K(x)=\max_{y\in K}\langle x,y\rangle,
\qquad x\in\mathbb R^n.
\]
For convex bodies $K$ and $L$ we have
\[
h_{K+L}=h_K+h_L.
\]

A {zonotope} is a finite Minkowski sum of line segments.  Every
origin-symmetric zonotope can therefore be written in the form
\[
Z=\sum_{j=1}^m[-v_j,v_j],
\]
where $v_1,\dots,v_m\in\mathbb R^n$, and its support function is
\[
h_Z(x)=\sum_{j=1}^m|\langle x,v_j\rangle|.
\]

A {zonoid} is a Hausdorff limit of zonotopes.  An
origin-symmetric convex body $Z$ is a zonoid if and only if there is a
finite even Borel measure $\mu$ on $S^{n-1}$ such that
\[
h_Z(x)
=
\int_{S^{n-1}}
|\langle x,\theta\rangle|\,d\mu(\theta),
\qquad x\in\mathbb R^n;
\]
see \cite[Theorem~3.5.3]{Schneider}.

A polytope is a zonoid if and only if it is a zonotope \cite[Corollary~3.5.7]{Schneider}.

If $K\subset\mathbb R^n$ is a convex body containing the origin in its
interior, its polar is
\[
K^\circ
=
\{x\in\mathbb R^n:
\langle x,y\rangle\leq1
\text{ for every }y\in K\},
\]
and
\[
\|x\|_{K^\circ}=h_K(x).
\]
Thus, by \eqref{eq:p-integral-representation}, an origin-symmetric convex body $K$ is a zonoid if and only if $K^\circ$
is the unit ball of a finite-dimensional normed space that embeds in  $L_1$.

 \section{Main results}

If $K$ is an origin-symmetric convex body, then $0\in\operatorname{int}K$. Thus, 
 for every $p>-n$, the function $h_K^p$ is locally integrable and defines a
 homogeneous tempered distribution of degree $p$. For $p\neq0$, set
 \[
 F_p(x)=\frac{1}{p}h_K^p(x).
 \]

 An equivalent form of our main result is the following.
 
 \begin{theorem}\label{thm:main}
 	Let $n\geq3$, let $K\subset\R^n$ be an origin-symmetric
 	full-dimensional polytope, and let
 	\[
 	-n+3<p<1,\qquad p\neq0.
 	\]
 	If $-\widehat{F_p}$ is a positive distribution outside the origin, then
 	$K$ is a zonotope.
 \end{theorem}
 
  Theorem~\ref{thm:main} yields the following classification.
 
 \begin{corollary}\label{cor:polyhedral}
 	Let $n\geq3$, let $L\subset\R^n$ be an origin-symmetric
 	full-dimensional polytope, and let $-n+3<p<1$. Then the following are
 	equivalent:
 	\begin{enumerate}
 		\item $L\in\mathcal I_p$;
 		\item $L^\circ$ is a zonotope;
 		\item $L\in\mathcal I_1$.
 	\end{enumerate}
 \end{corollary}
 
 \begin{proof}
 	The equivalence of (2) and (3) was discussed in the previous section. The 
 	inclusion
 	\[
 	\mathcal I_1\subset\mathcal I_p,
 	\qquad -n<p<1,
 	\]
 gives (3)~$\Rightarrow$~(1).
 	
 It remains to prove \textup{(1)}~$\Rightarrow$~\textup{(2)}. Put
 \(K=L^\circ\), so that \(h_K=\|\cdot\|_L\), and suppose first that
 \(p\neq0\). By the Fourier characterizations of embeddings into
 \(L_p\),
 \[
 -\widehat{F_p}
 =
 -\frac{1}{p}\widehat{h_K^p}
 \]
 is a positive distribution on \(\mathbb R^n\setminus\{0\}\). 
 Therefore, Theorem~\ref{thm:main} implies that \(K\) is a zonotope.

  Finally, suppose that $p=0$. This case
 	can occur only when $n\geq4$. 
 	Since $L\in\mathcal I_0$, by \cite{KaltonKoldobskyYaskinYaskina}, we have 
 	\[
 	L\in \mathcal I_q,
 	\]
 	for any $
 	-n+3<q<0. $ Thus, we are in the case discussed above, and so $K=L^\circ$ is a
 	zonotope. 
 \end{proof}

 \begin{corollary}\label{cor:k-intersection}
 	Let $n\geq5$ and let $1\leq k<n-3$ be an integer. An
 	origin-symmetric full-dimensional polytope $L\subset\mathbb R^n$ is a
 	$k$-intersection body if and only if $L^\circ$ is a zonotope.
 \end{corollary}
 
 \begin{proof}
 	This follows from Corollary~\ref{cor:polyhedral} and the
 	Fourier-analytic characterization of $k$-intersection bodies.
 \end{proof}

 \section{Proof of Theorem \ref{thm:main}}

  For a polytope $K\subset\R^n$ and $y\in\R^n$, let
 \[
 K^y=\{z\in K:\langle z,y\rangle=h_K(y)\}
 \]
 be the face of $K$ exposed by $y$. If $F$ is a face of $K$, its normal
 cone is
 \[
 N_K(F)=\{y\in\R^n:F\subset K^y\}.
 \]
 For a linear subspace $V\subset\R^n$, $P_V$ denotes the orthogonal
 projection onto $V$.
 
 Let the vertices of the origin-symmetric full-dimensional polytope $K$
 be
 \[
 v_1,\ldots,v_N,
 \]
 and put $E_i=N_K(v_i)$. Then
 \[
 h_K(x)=\langle x,v_i\rangle,\qquad x\in E_i.
 \]
 If $F_{ij}=[v_i,v_j]$ is an edge, put
 \[
 l_{ij}=|v_j-v_i|,
 \qquad
 g_{ij}=\frac{v_j-v_i}{|v_j-v_i|},
 \qquad
 E_{ij}=N_K(F_{ij})=E_i\cap E_j.
 \]
 Thus $g_{ij}$ is a unit vector parallel to the edge and
 $E_{ij}\subset g_{ij}^\perp$.  Every sum over edges is taken over
 unordered edges, each counted once. We will often refer to $E_{ij}$ as a wall, since it separates the cones $E_i$ and $E_j$.
 
 For integrals over a Euclidean space, a linear or affine subspace, or a
 polyhedral cone in its linear span, $dx$ denotes the corresponding
 Lebesgue measure. Thus, on an $m$-dimensional affine set, it agrees with
 $m$-dimensional Hausdorff measure. For integrals over Lipschitz
 hypersurfaces, we use $d\sigma(x)$ for surface measure.  
 
 For the reader's convenience, the proof is split into several steps.
 
\noindent\textbf{Step 1.}
  We compute  the second directional
 derivative of $F_p$ and isolate the edge contributions. Since $h_K$ is piecewise linear, its gradient is
 constant on each $E_i$ and can jump only across the codimension-one walls
 $E_{ij}$ corresponding to edges.
 
 \begin{lemma}\label{lem:hessian}
 	Let $-n+2<p<1$, $p\neq0$, and let $g$ be a unit vector. Then, for every
 	$\phi\in\mathcal S(\R^n)$,
 	\begin{equation}\label{eq:hessian-action}
 		\begin{aligned}
 			\left\langle\partial_g^2F_p,\phi\right\rangle
 			={}&\sum_{F_{ij}\text{ edge}}
 			l_{ij}\langle g_{ij},g\rangle^2
 			\int_{E_{ij}}h_K^{p-1}(x)\phi(x)\,dx\\
 			&+(p-1)\sum_{i=1}^N\langle v_i,g\rangle^2
 			\int_{E_i}\langle x,v_i\rangle^{p-2}\phi(x)\,dx.
 		\end{aligned}
 	\end{equation}
 \end{lemma}
 
 \begin{proof}
 	It is enough to prove the formula for test functions supported away from the
 	origin. Indeed, both sides of the asserted equality are homogeneous
 	distributions of degree $p-2>-n$. Thus, they cannot contain  distributions supported at the
 	origin, since any such distribution is a finite linear combination of derivatives of the delta function, whose homogeneity degrees are at most $-n$.

Let 
 $
 	\phi\in C_c^\infty(\R^n\setminus\{0\}),
 	$  and choose $0<r<R$ so that
 	\[
 	\operatorname{supp}\phi\subset B_R\setminus\overline{B_r}.
 	\]
 	Define
 	\[
 	D_i=E_i\cap\bigl(B_R\setminus\overline{B_r}\bigr)\]
 	and \[
 	f_i(x)=\frac1p\langle x,v_i\rangle^p,\qquad x\in E_i.
 	\]
 	Note that $f_i$ is the restriction of $F_p$ to $E_i$. The sets $D_i$ are bounded Lipschitz domains. Since the interiors of the
 	cones $E_i$ are pairwise disjoint and their boundaries have zero Lebesgue
 	measure, for $1\le \alpha,\beta\le n$, we have
 	\begin{equation}\label{eq:coordinate-start}
 		\left\langle \partial_\alpha\partial_\beta F_p,\phi\right\rangle
 		=\left\langle F_p,\partial_\alpha\partial_\beta\phi\right\rangle
 		=\sum_{i=1}^N\int_{D_i}
 		f_i(x)\,\partial_\alpha\partial_\beta\phi(x)\,dx.
 	\end{equation}
 	
 	We now integrate by parts twice on a fixed $D_i$. Let $n_i$ denote the outer
 	unit normal to $D_i$. Applying the divergence theorem to
 	$f_i\,\partial_\beta\phi\,e_\alpha$ gives
 	\begin{equation}\label{eq:first-ibp-hessian}
 		\int_{D_i}f_i\,\partial_\alpha\partial_\beta\phi\,dx
 		=\int_{\partial D_i}f_i(x)\,\partial_\beta\phi(x)\,(n_i)_\alpha
 		\,d\sigma(x)
 		-\int_{D_i}\partial_\alpha f_i(x)\,\partial_\beta\phi(x)\,dx.
 	\end{equation}
 	Next apply the divergence theorem to
 	$\partial_\alpha f_i(x)\phi(x)\,e_\beta$. This yields
 	\begin{equation}\label{eq:second-ibp-hessian}
 		\int_{D_i}\partial_\alpha f_i(x)\,\partial_\beta\phi(x)\,dx
 		=\int_{\partial D_i}\partial_\alpha f_i(x)\phi(x)\,(n_i)_\beta
 		\,d\sigma(x)
 		-\int_{D_i}\partial_\alpha\partial_\beta f_i(x)\,\phi(x)\,dx.
 	\end{equation}
 	Substituting \eqref{eq:second-ibp-hessian} into \eqref{eq:first-ibp-hessian}, we obtain
 	\begin{align*}
 		\int_{D_i}f_i(x)\,\partial_\alpha\partial_\beta\phi(x)\,dx
 		={}&\int_{D_i}\partial_\alpha\partial_\beta f_i(x)\,\phi(x)\,dx
 		\notag\\
 		&+\int_{\partial D_i}
 		\left(
 		f_i(x)(n_i)_\alpha\partial_\beta\phi(x)
 		-\partial_\alpha f_i(x)(n_i)_\beta\phi(x)
 		\right)\,d\sigma(x).                    
 	\end{align*}
 	Thus, \eqref{eq:coordinate-start} becomes
 	 	\begin{align}
  		\left\langle \partial_\alpha\partial_\beta F_p,\phi\right\rangle 
 		={}&\sum_{i=1}^N\int_{D_i}\partial_\alpha\partial_\beta f_i(x)\,\phi(x)\,dx
 		\notag\\
 		&+\sum_{i=1}^N\int_{\partial D_i}
 		\left(
 		f_i(x)(n_i)_\alpha\partial_\beta\phi(x)
 		-\partial_\alpha f_i(x)(n_i)_\beta\phi(x)
 		\right)\,d\sigma(x).    \label{eq:coordinate-start-new}        
 	\end{align}
 	Because $\phi$ and all its derivatives vanish near $\partial B_r$ and
 	$\partial B_R$, the spherical parts of $\partial D_i$ make no contribution.
 	Thus only the codimension-one walls shared by adjacent normal cones remain.
 	
 	Fix an edge $F_{ij}=[v_i,v_j]$. Recall that
 	\[
 	g_{ij}=\frac{v_j-v_i}{|v_j-v_i|}
 	\]
 	points from $E_i$ to $E_j$. Hence the outer normal of $D_i$ along
 	$E_{ij}$ is $g_{ij}$, whereas the outer normal of $D_j$ is
 	$-g_{ij}$. On
 	$E_{ij}$ one has
 	\[
 	\langle x,v_i\rangle=\langle x,v_j\rangle=h_K(x),
 	\]
 	and hence $f_i=f_j$. Therefore the boundary terms involving
 	$\partial_\beta\phi$ cancel. The remaining terms combine to
 	\[
 	\bigl(\partial_\alpha f_j(x)-\partial_\alpha f_i(x)\bigr)
 	(g_{ij})_\beta\phi(x).
 	\]
 	Since
 	\[
 	\partial_\alpha f_i(x)
 	=\langle x,v_i\rangle^{p-1}(v_i)_\alpha,
 	\]
 	we obtain on $E_{ij}$
 	\begin{equation*}
 		\partial_\alpha f_j(x)-\partial_\alpha f_i(x)
 		=h_K^{p-1}(x)
 		\bigl((v_j)_\alpha-(v_i)_\alpha\bigr)=l_{ij}h_K^{p-1}(x)(g_{ij})_\alpha.
 	\end{equation*}
 	Consequently, the total contribution of the wall $E_{ij}$ is
 	\[
 	l_{ij}(g_{ij})_\alpha(g_{ij})_\beta
 	\int_{E_{ij}}h_K^{p-1}(x)\phi(x)\,dx.
 	\]
 	Every codimension-one wall is counted once. Intersections of two or more
 	walls have zero $(n-1)$-dimensional Lebesgue measure and produce no
 	additional term.
 	
 	Inside $E_i$, ordinary differentiation gives
 	\[
 	\partial_\alpha\partial_\beta f_i(x)
 	=(p-1)(v_i)_\alpha(v_i)_\beta
 	\langle x,v_i\rangle^{p-2}.
 	\]
 	Substituting all such terms into
 	\eqref{eq:coordinate-start-new}, we find that, for every
 	$\phi\in C_c^\infty(\R^n\setminus\{0\})$,
 	\begin{align}
 		\left\langle\partial_\alpha\partial_\beta F_p,\phi\right\rangle
 		={}&\sum_{F_{ij}\text{ edge}}
 		l_{ij}(g_{ij})_\alpha(g_{ij})_\beta
 		\int_{E_{ij}}h_K^{p-1}(x)\phi(x)\,dx \notag\\
 		&+(p-1)\sum_{i=1}^N
 		(v_i)_\alpha(v_i)_\beta
 		\int_{E_i}\langle x,v_i\rangle^{p-2}\phi(x)\,dx.
 		\label{eq:coordinate-hessian-away}
 	\end{align}
 	
 	By the homogeneity argument at the beginning of the proof, this identity
 	extends to every $\phi\in C_c^\infty(\R^n)$. Finally,
 	\[
 	\partial_g^2=\sum_{\alpha,\beta=1}^n
 	g_\alpha g_\beta\partial_\alpha\partial_\beta.
 	\]
 	Multiplying \eqref{eq:coordinate-hessian-away} by
 	$g_\alpha g_\beta$ and summing over $\alpha,\beta$ gives
 	\eqref{eq:hessian-action} for every
 	$\phi\in C_c^\infty(\R^n)$.  Since $C_c^\infty(\R^n)$ is dense in $\mathcal S(\R^n)$, the identity holds for every
 	$\phi\in\mathcal S(\R^n)$.
 \end{proof}

\noindent\textbf{Step 2.} In formula \eqref{eq:hessian-action}, we isolate the terms that correspond to the edges parallel to a fixed vector.
 
Recall that
 \[
 n\geq3,\qquad -n+3<p<1,\qquad p\neq0,
 \]
 and that $-\widehat{F_p}$ is positive outside the origin. Fix a unit
 vector $g$ parallel to an edge of $K$, and put
 \[
 H=g^\perp,
 \qquad
 Q=P_HK.
 \]
 For $y\in H$, one has $h_K(y)=h_Q(y)$. Define
 \begin{equation}\label{eq:Tg}
 	T_g(y)=\sum_{F_{ij}\parallel g}
 	l_{ij}h_Q^{p-1}(y)\ind_{E_{ij}}(y),
 	\qquad y\in H.
 \end{equation}
 The function $T_g$ is homogeneous of degree $p-1$ on the
 $(n-1)$-dimensional subspace $H$. Since $p-1>-(n-1)$, it is locally integrable
 at the origin; it has at most polynomial growth at infinity and hence
 defines a tempered distribution on $H$.
 
 \begin{lemma}\label{prop:Tg}
 	The distribution $T_g$ is positive definite on $H$.
 \end{lemma}
 
 \begin{proof}
 	Put
 	\[
 	U=\partial_g^2F_p.
 	\]
 	The Fourier transform of $U$ is positive outside the origin. Indeed, as a
 	distribution,
 	\[
 	\widehat U
 	=-\langle\,\cdot\,,g\rangle^2\widehat{F_p}
 	=\langle\,\cdot\,,g\rangle^2\bigl(-\widehat{F_p}\bigr).
 	\]

 	Let $\psi\in C_c^\infty(H)$ be non-negative, and choose
 	$\chi\in C_c^\infty((-1,1))$ such that
 	\[
 	\chi\geq0,
 	\qquad
 	\int_{\mathbb R}\chi(\tau)\,d\tau=1.
 	\]
 	For $t\in\mathbb R$, define
 	\[
 	\Phi_t(\tau g+\eta)
 	=
 	\chi(\tau-t)\psi(\eta),
 	\qquad
 	\tau\in\mathbb R,\quad \eta\in H.
 	\]
 	Then $\Phi_t\in C_c^\infty(\mathbb R^n)$ and $\Phi_t\geq0$.
 	Moreover, for $|t|>1$,
 	\[
 	0\notin\operatorname{supp}\Phi_t.
 	\]
 	Since $\widehat U$ is positive on
 	$\mathbb R^n\setminus\{0\}$, for $|t|>1$ we  have
 	\begin{equation}\label{eq:positive-U-Phi}
 		0\leq\langle\widehat U,\Phi_t\rangle
 		=
 		\langle U,\widehat{\Phi_t}\rangle.
 	\end{equation}

 	Write
 	\[
 	x=s g+y,
 	\qquad s\in\R,\quad y\in H.
 	\]
 	A direct calculation gives
 	\begin{equation}\label{eq:Phi-hat}
 		\widehat{\Phi_t}(s g+y)
 		=e^{-its}\widehat\chi(s)\widehat\psi(y).
 	\end{equation}
 	Applying Lemma~\ref{lem:hessian} with
 	$\phi=\widehat{\Phi_t}$, we obtain
 	\begin{align}
 		0\leq\langle\widehat U,\Phi_t\rangle
 		={}&\sum_{F_{ij}\text{ edge}}
 		l_{ij}\langle g_{ij},g\rangle^2
 		\int_{E_{ij}}h_K^{p-1}(x)\widehat{\Phi_t}(x)
 		\,dx \notag\\
 		&+(p-1)\sum_{i=1}^N\langle v_i,g\rangle^2
 		\int_{E_i}\langle x,v_i\rangle^{p-2}
 		\widehat{\Phi_t}(x)\,dx.                 \label{eq:direct-first-descent}
 	\end{align}
 	
 	Suppose first that $F_{ij}$ is parallel to $g$. Then
 	$g_{ij}=\pm g$ and $E_{ij}\subset g^\perp=H$. Thus $s=0$ on
 	$E_{ij}$, and \eqref{eq:Phi-hat} gives
 	\[
 	\widehat{\Phi_t}(y)
 	=\widehat\chi(0)\widehat\psi(y)
 	=\widehat\psi(y),
 	\qquad y\in E_{ij},
 	\]
 	because $\widehat\chi(0)=\int_\R\chi(\tau)\, d\tau=1$. Consequently, the total
 	contribution in \eqref{eq:direct-first-descent} of the edges parallel to
 	$g$ is
 	\begin{equation}
 		\sum_{F_{ij}\parallel g}l_{ij}
 		\int_{E_{ij}}h_Q^{p-1}(y)\widehat\psi(y)\,dy 
 	=\langle T_g,\widehat\psi\rangle
 		=\langle\widehat{T_g},\psi\rangle.       \label{eq:surviving-first-descent}
 	\end{equation}
 	
 	We next show that every other term in
 	\eqref{eq:direct-first-descent} tends to zero as $|t|\to\infty$.
 	Consider an edge $F_{ij}$ not parallel to $g$, and put
 	\[
 	V_{ij}=g_{ij}^\perp,
 	\qquad
 	a_{ij}=P_{V_{ij}}g.
 	\]
 	Since $F_{ij}$ is not parallel to $g$, one has $a_{ij}\neq0$. For
 	$x\in V_{ij}$,
 	\[
 	\langle x,g\rangle=\langle x,a_{ij}\rangle.
 	\]
 	Define
 	\[
 	A_{ij}(x)=\ind_{E_{ij}}(x)h_K^{p-1}(x)
 	\widehat\chi(\langle x,g\rangle)
 	\widehat\psi\bigl(P_Hx\bigr),
 	\qquad x\in V_{ij}.
 	\]
 	Then the integral over $E_{ij}$ in \eqref{eq:direct-first-descent} is
 	\begin{equation}\label{eq:first-descent-wall-Fourier}
 		\int_{V_{ij}}e^{-it\langle x,a_{ij}\rangle}A_{ij}(x)\,dx
 		=\widehat{A_{ij}} (t a_{ij}),
 	\end{equation}
 	where the Fourier transform on the right is taken in  
 	$V_{ij}$. 
 	
 	Our goal is to show that the latter tends to zero as $|t|\to \infty$.
 	Observe that $A_{ij}$ is integrable on $V_{ij}$. Indeed, near the origin,
 	\[
 	|A_{ij}(x)|\leq C|x|^{p-1},
 	\]
 	which is integrable since $p-1>-(n-1)$. 
   At infinity, $A_{ij}$ is rapidly decreasing because of the Schwartz functions $\widehat\chi$ and $\widehat\psi$. We will provide more details below, since this argument will be used a few times.
 	
 	Since $\widehat\chi\in\mathcal S(\mathbb R)$ and
 	$\widehat\psi\in\mathcal S(H)$, for every $N>0$, we have
 	\begin{align*}
 	\left|
 	\widehat\chi(\langle x,g\rangle)
 	\widehat\psi(P_Hx)
 	\right|
 	&\leq
 	C_N(1+|\langle x,g\rangle|^2)^{-N}
 	(1+|P_Hx|^2)^{-N}\\
 	  	& 	  \leq	C_N(1+|\langle x,g\rangle|^2+|P_Hx|^2 )^{-N}
 	 \\
 	&= C_N(1+|x|^2)^{-N}.
 	\end{align*}
 Thus this product is rapidly
 	decreasing at infinity, and hence
 	$A_{ij}\in L^1(V_{ij})$. Since $a_{ij}\neq0$, the
 	Riemann--Lebesgue lemma applied in $V_{ij}$ to
 	\eqref{eq:first-descent-wall-Fourier} shows that this  integral tends
 	to zero as $|t|\to \infty$.
 	
 	Now consider the integral over $E_i$ in \eqref{eq:direct-first-descent}. Let
 	\[
 	A_i(x)=\ind_{E_i}(x)\langle x,v_i\rangle^{p-2}
 	\widehat\chi(\langle x,g\rangle)\widehat\psi(P_Hx).
 	\]
 	The corresponding integral becomes
 	\[
 	\int_{\R^n}e^{-it\langle x,g\rangle}A_i(x)\,dx
 	=\widehat{A_i}(tg).
 	\]
 	Again $A_i\in L^1(\R^n)$. Near the origin this follows from
	\[
|A_{i}(x)|\leq C|x|^{p-2},
\]
 	and at infinity it follows from the Schwartz factors. The
 	Riemann--Lebesgue lemma gives
 	\[
 	\widehat{A_i}(tg)\to 0,
 	\quad \mbox{as }|t|\to\infty.
 	\]
 Thus every term in \eqref{eq:direct-first-descent}, except for the
 sum of the  terms corresponding to edges $F_{ij}$ parallel to $g$, tends
 to zero as $|t|\to\infty$.
 	
 	Combining this with \eqref{eq:surviving-first-descent}, we obtain
 	\[
 	\langle\widehat U,\Phi_t\rangle
 	=\langle\widehat{T_g},\psi\rangle+o(1), \quad \mbox{as }|t|\to\infty.
 	\]
 	The left-hand side is non-negative for every $|t|>1$ by
 	\eqref{eq:positive-U-Phi}. Passing to the limit gives
 	\[
 	\langle\widehat{T_g},\psi\rangle\geq0,
 	\]
 for every non-negative $\psi\in C_c^\infty(H)$. Since every non-negative $\varphi\in\mathcal S(H)$ can be approximated  by non-negative functions in $C_c^\infty(H)$,
 	we get 
 	\[
 	\langle \widehat{T_g},\varphi\rangle\geq0.
 	\]
 	Thus $T_g$ is positive definite on $H$.
 \end{proof}
 
 \bigskip

 Let $a_1,\ldots,a_M$ be the vertices of $Q=P_HK$, and put
 \[
 C_i=N_Q(a_i)\subset H.
 \]
 The fiber above $a_i$ is
 \[
 J_i=K\cap(a_i+\R g).
 \]
 It is either a point or an edge parallel to $g$. Set
 \[
 \lambda_i=\length(J_i).
 \]
 
 \begin{lemma}\label{lem:fibres}
 	For $y\in\relint C_i$,
 	\begin{equation}\label{eq:T-chambers}
 		T_g(y)=\lambda_i h_Q^{p-1}(y).
 	\end{equation}
 \end{lemma}
 
 \begin{proof}
 	Let $y\in\relint C_i$. The face of $Q$ exposed by $y$ is the vertex
 	$a_i$. Since $y\in H=g^\perp$, the face of $K$ exposed by $y$ is exactly
 	the fiber $J_i$. Indeed, the projection of $K^y$ is contained in
 	$Q^y=\{a_i\}$, while every $x\in J_i$ satisfies
 	\[
 	\langle x,y\rangle
 	=\langle a_i,y\rangle
 	=h_Q(y)
 	=h_K(y).
 	\]
 	
 	Now consider an edge $F_{jm}\parallel g$ appearing in
 \eqref{eq:Tg}. By the definition of its normal cone,
 \[
 y\in E_{jm}=N_K(F_{jm})
 \quad\Longleftrightarrow\quad
 F_{jm}\subset K^y=J_i.
 \]
 If $J_i$ is a point, no edge can be contained in $J_i$. If $J_i$ is
 an edge, the only edge of $K$ contained in $J_i$ is $J_i$ itself.
 Consequently,
 \[
 \sum_{F_{jm}\parallel g}
 l_{jm}\ind_{E_{jm}}(y)
 =
 \length(J_i)
 =
 \lambda_i.
 \]
 Using \eqref{eq:Tg}, we conclude that
 \[
 T_g(y)
 =
 \lambda_i h_Q^{p-1}(y),
 \]
 as claimed.
 \end{proof}

 \bigskip

\noindent\textbf{Step 3.}  We now show that the numbers $\lambda_i$ in \eqref{eq:T-chambers} are
 all equal. The following lemma is stated in an arbitrary dimension $d$.
 
 \begin{lemma}\label{lem:no-jumps}
 	Let $Q\subset\R^d$ be a full-dimensional polytope with
 	$0\in\operatorname{int}Q$, let $h=h_Q$, and let
 	$C_i=N_Q(a_i)$ be its vertex normal cones. Suppose that
 	$\lambda_i\in\R$ and that the locally integrable function $T$ satisfies
 	\[
 	T(y)=\lambda_i h^{-\alpha}(y), 
 	\qquad y\in\relint C_i,
 	\]
 	where $0<\alpha<d-1$.
 	If $T$ is positive definite, then all the
 	numbers $\lambda_i$ are equal.
 \end{lemma}
 
 \begin{proof}
 	On each cone $C_i$,
 	\[
 	h(y)=\langle y,a_i\rangle,
 	\]
 	and, since $0\in\operatorname{int}Q$, one has $h^{-\alpha}(y)\le C |y|^{-\alpha}$.
 	Since  $\alpha<d-1$, all integrals considered below are convergent near the origin.

 		Fix a unit vector $u$ parallel to an edge of $Q$. Choose an orientation
 		of every edge $[a_i,a_j]$ of $Q$; for the edges parallel to $u$, choose
 		the orientation so that
 		\[
 		a_j-a_i=c_{ij}u,
 		\qquad c_{ij}>0.
 		\]
 		The orientations of all other edges may be chosen arbitrarily.
 		
 		For every oriented edge, put
 		\[
 		\nu_{ij}=\frac{a_j-a_i}{|a_j-a_i|}.
 		\]
 		Then $\nu_{ij}$ is the unit normal to the common wall
 		\[
 		C_{ij}=C_i\cap C_j=N_Q([a_i,a_j])
 		\subset\nu_{ij}^{\perp},
 		\]
 		pointing from $C_i$ to $C_j$.  		
 		For an edge parallel to $u$, our choice of orientation gives
 		$\nu_{ij}=u$.
 		
 		We first compute the distributional derivative of $T$. Let
 		$\varphi\in\mathcal S(\mathbb R^d)$. By definition,
 		\[
 		\langle\partial_uT,\varphi\rangle
 		=-	\langle T,\partial_u\varphi\rangle
 		=
 		-\sum_i\lambda_i
 		\int_{C_i}h^{-\alpha}(y)\partial_u\varphi(y)\,dy.
 		\]
 		For $0<r<R$, put
 		\[
 		D_i=C_i\cap
 		\bigl(B_R\setminus\overline{B_r}\bigr).
 		\]
 		On $C_i$,
 		\[
 		\partial_u h^{-\alpha}(y)
 		=
 		-\alpha h^{-\alpha-1}(y)\langle a_i,u\rangle.
 		\]
 		Applying the divergence theorem to
 		$\lambda_ih^{-\alpha}\varphi u$ on $D_i$, summing over $i$, and then
 		letting $r\to0$ and $R\to\infty$, we obtain
 		\begin{align}
 			\langle\partial_uT,\varphi\rangle
 			={}&
 			-\alpha\sum_i\lambda_i
 			\int_{C_i}h^{-\alpha-1}(y)
 			\langle a_i,u\rangle\varphi(y)\,dy
 			\notag\\
 			&+
 			\sum_{[a_i,a_j]\text{ edge}}
 			(\lambda_j-\lambda_i)\langle\nu_{ij},u\rangle
 			\int_{C_{ij}}h^{-\alpha}(y)\varphi(y)\,dy.
 			\label{eq:full-first-derivative}
 		\end{align}
 		Indeed, along $C_{ij}$ the outer normals of $C_i$ and $C_j$ are
 		$\nu_{ij}$ and $-\nu_{ij}$, respectively. The inner spherical boundary
 		terms are
 		\[
 		O(r^{d-1-\alpha})=o(1),
 		\]
 		and the outer spherical boundary terms tend to zero by the rapid decay
 		of $\varphi$.
 		
 		Put $H=u^\perp$. For $\psi\in\mathcal S(H)$, define
 		\begin{equation}\label{eq:Ju}
 			\langle J_u,\psi\rangle
 			=
 			\sum_{[a_i,a_j]\parallel u}
 			(\lambda_j-\lambda_i)
 			\int_{C_{ij}}h^{-\alpha}(y)\psi(y)\,dy.
 		\end{equation}
 		Since $\alpha<d-1$, this defines a tempered distribution on $H$.
 		
 		Let $\mathcal R_u\in\mathcal S'(\mathbb R^d)$ contain all the remaining
 		terms in \eqref{eq:full-first-derivative}, that is,
 		\begin{align*}
 			\langle\mathcal R_u,\varphi\rangle
 			={}&
 			-\alpha\sum_i\lambda_i
 			\int_{C_i}h^{-\alpha-1}(y)
 			\langle a_i,u\rangle\varphi(y)\,dy\\
 			&+
 			\sum_{\substack{[a_i,a_j]\text{ edge}\\
 					[a_i,a_j]\not\parallel u}}
 			(\lambda_j-\lambda_i)\langle\nu_{ij},u\rangle
 			\int_{C_{ij}}h^{-\alpha}(y)\varphi(y)\,dy.
 		\end{align*}
 		For an edge parallel to $u$, we have $\nu_{ij}=u$ and
 		$C_{ij}\subset H$. Hence
 		\begin{equation}\label{eq:second-descent-decomposition}
 			\langle\partial_uT,\varphi\rangle
 			=
 			\langle J_u,\varphi|_H\rangle
 			+
 			\langle\mathcal R_u,\varphi\rangle.
 		\end{equation}

 	We claim that $J_u=0$. Let $\psi\in C_c^\infty(H)$ be non-negative.
 	Choose $\chi\in C_c^\infty((-1,1))$ such that
 	\[
 	\chi\geq0,
 	\qquad
 	\int_{\mathbb R}\chi(s)\,ds=1.
 	\]
 	For $x\in\mathbb R^d$, write
 	\[
 x=s u+\eta,
 	\qquad s\in\mathbb R,\quad \eta\in H=u^\perp,
 	\]
 	so that
 	\[
 	s=\langle x,u\rangle.
 	\]
 	For $t\in\mathbb R$, define
 	\[
 	\Phi_t(x)
 	=
 	\chi(s-t)\psi(\eta).
 	\]

 	 Using properties of the Fourier transform, we obtain
 		\begin{equation}\label{eq:second-descent-positive}
 		\left\langle
 		\widehat T,\langle\,\cdot\,,u\rangle\Phi_t
 		\right\rangle
 		= 	\left\langle
 		 T,\widehat{\langle\,\cdot\,,u\rangle\Phi_t}
 		\right\rangle
	=
\left\langle
T, i\partial_u \widehat{\Phi_t}
\right\rangle
 		=
 		\left\langle
 		-i\partial_uT,\widehat{\Phi_t}
 		\right\rangle.
 	\end{equation}
  	Since $\operatorname{supp}\chi\subset(-1,1)$, on
 	$\operatorname{supp}\Phi_t$ we have
 	\[
 	|s-t|<1.
 	\]
 	Hence $s=\langle x,u\rangle>0$ when $t>1$, whereas $s=\langle x,u\rangle<0$ when $t<-1$.
 	Since $\Phi_t\geq0$ and $\widehat T$ is a positive distribution, it
 	follows that
 	\begin{align*}
 		\left\langle
 		\widehat T,\langle\,\cdot\,,u\rangle\Phi_t
 		\right\rangle
 		\geq0,&
 		\quad \mathrm{ when }\,\, t>1,\\
 		 	\left\langle
 		\widehat T,\langle\,\cdot\,,u\rangle\Phi_t
 		\right\rangle
 		\leq0,&
 		\quad \mathrm{ when }\,\,  t<-1.
 	\end{align*}

 		Writing $y=\tau u+z$, with $\tau\in \mathbb R$ and  $z\in H$, gives
 		\begin{equation*}
 			\widehat{\Phi_t}(\tau u+z)
 			=
 			e^{-it\tau}\widehat\chi(\tau)\widehat\psi(z).
 		\end{equation*}
 		Since $\widehat\chi(0)=1$,
 		\[
 		\widehat{\Phi_t}|_H=\widehat\psi.
 		\]
 		Applying \eqref{eq:second-descent-decomposition} to
 		$\varphi=\widehat{\Phi_t}$, we obtain
 		\[
 		\langle\partial_uT,\widehat{\Phi_t}\rangle
 		=
 		\langle J_u,\widehat\psi\rangle
 		+
 		\langle\mathcal R_u,\widehat{\Phi_t}\rangle.
 		\]
 		Multiplying by $-i$ and using
 		\eqref{eq:second-descent-positive}, we get
 		\begin{equation}\label{mult_by-i}
 		\left\langle
 		\widehat T,\langle\cdot,u\rangle\Phi_t
 		\right\rangle
 		=
 		-i\langle J_u,\widehat\psi\rangle
 		-
 		i\langle\mathcal R_u,\widehat{\Phi_t}\rangle.
 		\end{equation}

 		We next show that
 		\begin{equation}\label{eq:second-remainder-vanishes}
 			\langle\mathcal R_u,\widehat{\Phi_t}\rangle
 			\to0
 			\qquad\mbox{as } |t|\to\infty.
 		\end{equation} 		
 		The full-dimensional part of $\mathcal R_u$ has density
 		\[
 		B(y)
 		=
 		-\alpha\sum_i\lambda_i h^{-\alpha-1}(y)
 		\langle a_i,u\rangle\ind_{C_i}(y).
 		\]
 		Its action on $\widehat{\Phi_t}$ is
 		\[
 		\int_{\mathbb R^d}
 		e^{-it\langle y,u\rangle}A_0(y)\,dy
 		=
 		\widehat{A_0}(tu),
 		\]
 		where
 		\[
 		A_0(y)
 		=
 		B(y)
 		\widehat\chi(\langle y,u\rangle)
 		\widehat\psi(P_Hy).
 		\]
 		The function $A_0$  belongs to $ L^1(\mathbb R^d)$,  since $|A_0(y)|=O(|y|^{-\alpha-1})$ near the origin, and  it is  rapidly decreasing at infinity because of the two Schwartz factors. The
 		Riemann--Lebesgue lemma gives
 		\[
 		\widehat{A_0}(tu)\to0,
 		\qquad\mbox{as } |t|\to\infty.
 		\]
 		
 		Now consider the part of $\mathcal R_u$ that contains a wall
 		\[
 		C_{ij}\subset V_{ij}:=\nu_{ij}^{\perp}
 		\]
 		corresponding to an edge not parallel to $u$, and put
 		\[
 		b_{ij}=P_{V_{ij}}u.
 		\]
 		Then $b_{ij}\neq0$; otherwise $u$ would be parallel to $\nu_{ij}$, and
 		hence $[a_i,a_j]$ would be parallel to $u$.
 		
 		Put
 		$$
 			A_{ij}(y)
 			={}
 			(\lambda_j-\lambda_i)\langle\nu_{ij},u\rangle
 			\ind_{C_{ij}}(y)h^{-\alpha} (y)
 			\widehat\chi(\langle y,u\rangle)
 			\widehat\psi(P_Hy).
 $$
 		Since $y\in V_{ij}$,
 		\[
 		\langle y,u\rangle=\langle y,b_{ij}\rangle,
 		\]
 		and the corresponding wall contribution is
 		\begin{equation}\label{FTAij}
 		\int_{V_{ij}}
 		e^{-it\langle y,b_{ij}\rangle}A_{ij}(y)\,dy
 		=
 		\widehat{A_{ij}}(tb_{ij}),
 		\end{equation}
 		where the Fourier transform on the right is taken in  
 	$V_{ij}$.

 		Again, observe that $A_{ij}\in L^1(V_{ij})$.  Since $b_{ij}\neq0$, the
 		Riemann--Lebesgue lemma on $V_{ij}$ shows that \eqref{FTAij} tends to zero as $|t|\to \infty$.
 		There are only finitely many walls, and
 		\eqref{eq:second-remainder-vanishes} follows.
 		
 		Set
 		\[
 	A_t=
 	\left\langle
 	\widehat T,\langle\cdot,u\rangle\Phi_t
 	\right\rangle.
 		\]
 		By \eqref{mult_by-i} and
 		\eqref{eq:second-remainder-vanishes},
 		\[
 		A_t
 		=
 		-i\langle J_u,\widehat\psi\rangle+o(1),
 		\qquad\mbox{as }|t|\to\infty.
 		\]
 		The numbers $A_t$ are real. Moreover,
 		\[
 		A_t\geq0\quad(t>1),
 		\qquad
 		A_t\leq0\quad(t<-1).
 		\]
 		Taking the limits $t\to+\infty$ and $t\to-\infty$ therefore gives
 		\[
 		-i\langle J_u,\widehat\psi\rangle=0,
 		\]
 		and hence
 		\begin{equation}\label{eq:J}
 		\langle\widehat{J_u},\psi\rangle
 		=
 		\langle J_u,\widehat\psi\rangle
 		=
 		0
 		\end{equation}
 		for every non-negative $\psi\in C_c^\infty(H)$.
 		
 		Observe that every test function in $C_c^\infty(H)$ is a complex linear
 		combination of non-negative test functions. Indeed, if $\varphi\in C_c^\infty(H)$ is real-valued, choose
 		a non-negative $\zeta\in C_c^\infty(H)$ with $\zeta=1$ on the support of
 		$\varphi$, and take
 		$M\geq\|\varphi\|_\infty$. Then
 		\[
 		\varphi=(\varphi+M\zeta)-M\zeta,
 		\]
 		and both terms on the right are non-negative. Complex-valued test
 		functions are handled by separating real and imaginary parts.
 		
 		It follows from \eqref{eq:J} that
 		$\widehat{J_u}=0$, and hence, 
 		\(
 		J_u=0.
 		\)

 		We now show that the coefficients $\lambda_i$ are all equal. The relative interiors of the cones $C_{ij}$, as $[a_i,a_j]$ ranges
 		over the edges of $Q$ parallel to $u$, are pairwise disjoint.
 		Fix one such edge $[a_i,a_j]$ and choose a nonzero function
 		\[
 		\varphi\in C_c^\infty
 		\bigl(\operatorname{relint} C_{ij}\bigr),
 		\qquad \varphi\geq0.
 		\]
 		Since $J_u=0$, the definition \eqref{eq:Ju} gives
 		\[
 		0=\langle J_u,\varphi\rangle
 		=
 		(\lambda_j-\lambda_i)
 		\int_{C_{ij}}h(y)^{-\alpha}\varphi(y)\,dy.
 		\]
 		The integral on the right is strictly positive. Hence
 		\[
 		\lambda_i=\lambda_j.
 		\]
 		
 		Thus $\lambda_i=\lambda_j$ for every edge $[a_i,a_j]$ parallel to
 		$u$. Since $u$ was an arbitrary edge direction of $Q$, the same
 		argument applies to every edge of $Q$. Therefore
 		\[
 		\lambda_i=\lambda_j
 		\]
 		whenever $a_i$ and $a_j$ are joined by an edge. Since the edge graph
 		of $Q$ is connected, all the coefficients $\lambda_1,\ldots,\lambda_M$
 		are equal.

 \end{proof}

\noindent\textbf{Step 4.}
We show that every two-dimensional face of \(K\) is centrally symmetric, and hence $K$ is a zonotope.
Let \(F\) be a two-dimensional face of \(K\), let \(e\) be an edge of
\(F\), and choose a unit vector \(g\) parallel to \(e\). Set
\[
H=g^\perp,
\qquad
Q=P_HK.
\]
Let \(a_1,\ldots,a_M\) be the vertices of \(Q\), and let \(\lambda_i\)
denote the length of the fiber
\[
K\cap(a_i+\mathbb Rg),
\]
as in \eqref{eq:T-chambers}.

Applying Lemma~\ref{lem:no-jumps} with
\[
d=n-1,
\qquad
\alpha=1-p,
\]
and using
\[
-n+3<p<1
\quad\Longleftrightarrow\quad
0<\alpha<d-1,
\]
we obtain
\begin{equation}\label{eq:constant}
	\lambda_1=\cdots=\lambda_M.
\end{equation}

Choose \(y\in\operatorname{relint}N_K(F)\), so that \(K^y=F\).
Since \(g\) is parallel to \(F\), we have \(y\perp g\), and hence
\(y\in H\). Therefore
\[
Q^y=(P_HK)^y=P_H(K^y)=P_HF.
\]
Since  
\(\ker P_H=\mathbb Rg\),  
\(P_HF\) is one-dimensional, and therefore it is an edge of \(Q\).
Denote its endpoints by \(a\) and \(b\).

Since \(e\) is an edge of the polygon \(F\) parallel to \(g\),
projection along \(g\) maps \(e\) to one of the endpoints of the
segment \(P_HF\). Relabeling \(a\) and \(b\) if necessary, we may write
\[
e=F\cap(a+\mathbb Rg).
\]
For \(v\in\{a,b\}\), put
\[
J_v=K\cap(v+\mathbb Rg).
\]
If \(x\in J_v\), then \(P_Hx=v\in Q^y\). Since \(y\in H\),
\[
\langle x,y\rangle
=
\langle P_Hx,y\rangle
=
\langle v,y\rangle
=
h_Q(y)
=
h_K(y).
\]
Hence \(x\in K^y=F\), and therefore
\[
J_a,J_b\subset F.
\]
It follows that
\[
J_a=e,
\qquad
J_b=F\cap(b+\mathbb Rg).
\]
Since \(a\) and \(b\) are vertices of \(Q\), \eqref{eq:constant} gives
\[
\operatorname{length}(J_b)
=
\operatorname{length}(J_a)
=
\operatorname{length}(e)>0.
\]
Since \(b\) is an endpoint of \(P_HF\), the fiber
\(F\cap(b+\mathbb Rg)\) is a face of \(F\). As it has positive length,
it is an edge of \(F\). Thus \(J_b\) is an edge of \(F\), parallel to
\(e\), and of the same length.

Consequently, every edge of \(F\) has a parallel edge of the same
length. It follows that the edge vectors of \(F\), in cyclic order,
occur in opposite pairs. Hence all pairs of opposite vertices have
the same midpoint, and therefore \(F\) is centrally symmetric. 
Since \(F\) was arbitrary, every two-dimensional face of \(K\) is
centrally symmetric. By the classical characterization of zonotopes,
\(K\) is a zonotope; see \cite[Theorem~3.5.2]{Schneider}. This
completes the proof of Theorem~\ref{thm:main}.

\section{Application of Theorem~\ref{thm:main-intro} to hyperplane sections }\label{sec:section-lifting}

Let 
\(H\subset\mathbb R^n\) be a linear subspace. Recall that  \(\mathcal I_p(H)\) denotes the
class of convex bodies in  \(H\) that are unit balls of finite-dimensional normed spaces that embed in $L_p$, $p>-\mathrm{dim}\, H$. As was mentioned in the introduction, 
if $K$ is an origin-symmetric convex body in $\mathbb R^n$ that belongs to $\mathcal I_p$ for $p>-n+1$, then $K\cap H$ belongs to $\mathcal I_p(H)$ for each $(n-1)$-dimensional subspace $H\subset \mathbb R^n$. For $p>0$ this follows from the definition of embedding in $L_p$, for $p<0$ this was shown in \cite{Rubin} (see also \cite{Milman2006}), and for $p=0$ this can be obtained from the following  argument. If $K\in \mathcal I_0$, then $K\in \mathcal I_p$ for all $-n+1<p<0$. Therefore $K\cap H\in \mathcal I_p(H)$ for all $-n+1<p<0$ and all $(n-1)$-dimensional subspaces $H$. Hence,  $K\cap H\in \mathcal I_0(H)$ for   all $(n-1)$-dimensional subspaces $H$. For details,  see \cite{KaltonKoldobskyYaskinYaskina}.

On the other hand, it was shown by Neyman \cite{Neyman} 
that, for each $1\le p<\infty$, $p\ne 2$, there are $n$-dimensional normed spaces that do not embed in $L_p$, but all their
$(n-1)$-dimensional subspaces embed in $L_p$.  This result was extended to $\mathcal I_{-1}$ by 
Yaskina \cite{Yaskina}, who constructed an example of an origin-symmetric convex body in $\mathbb R^n$, $n\ge 5$, that is not an intersection body, but all its central sections are intersection bodies. The latter was further generalized to $k$-intersection bodies in \cite{Yaskin2008}.

 The goal of this section is to show that if $K$ is a polytope, then the reverse  implication also holds: $$K\cap H\in \mathcal I_p(H), \quad \forall H \implies K\in \mathcal I_p,$$
for an appropriate range of $p$.

Let
\(D^{\circ_H}\) denote the polar of a convex body \(D\subset H\), taken in
the Euclidean space \(H\). 
We first recall the relation between sections and projections under
polarity; see, for example, \cite[p.~22]{Gardner}.

\begin{lemma}\label{lem:section-polar-projection}
	Let \(K\subset\mathbb R^n\) be an origin-symmetric convex body, and let
	\(H\subset\mathbb R^n\) be a linear subspace. Then
	\[
	(K\cap H)^{\circ_H}=P_HK^\circ,
	\]
	where \(P_H\) is the orthogonal projection onto \(H\).
\end{lemma}

We will also need the following geometric fact.

\begin{lemma}\label{lem:hyperplane-projections-zonotope}
	Let \(K\subset\mathbb R^n\), \(n\geq4\), be a full-dimensional
	polytope. Suppose that \(P_HK\) is a zonotope for every $(n-1)$-dimensional subspace \(H\subset\mathbb R^n\). Then \(K\) is a zonotope.
\end{lemma}

\begin{proof}
	Let \(F\) be a two-dimensional face of \(K\), and let
	\[
	E=\operatorname{lin}(F-F)
	\]
		be the two-dimensional linear subspace parallel to the affine hull of $F$.
	Choose \(z\neq 0\) such that \(F=K^z\). 
	Since \(F=K^z\), the linear functional
	\(x\mapsto\langle x,z\rangle\) is constant on \(F\). Hence
	\(\langle x-y,z\rangle=0\) for all \(x,y\in F\), and therefore
	\(E\subset z^\perp\).
Because 
	\(
	\dim(E+\mathbb R z)=3
	\)
	and \(n\geq4\), there is an \((n-1)\)-dimensional subspace
	\(H\) containing \(E+\mathbb R z\).
	
	Since \(z\in H\), orthogonal projection commutes with taking the
	exposed face in direction \(z\):
	\[
	(P_HK)^z=P_H(K^z)=P_HF.
	\]
	Moreover, if \(x_0\in F\), then for every \(w\in E\),
	\[
	P_H(x_0+w)=P_Hx_0+w,
	\]
	because \(E\subset H\). Thus \(P_H\) restricts to a translation on
	\(\operatorname{aff}F=x_0+E\). Consequently, \(P_HF\) is a
	two-dimensional face of \(P_HK\) and is a translate of \(F\).
	
	By hypothesis, \(P_HK\) is a zonotope, so its two-dimensional face
	\(P_HF\) is centrally symmetric. Therefore \(F\) is centrally
	symmetric. 	
	Thus every two-dimensional face of \(K\) is centrally symmetric.
	By the classical characterization of zonotopes,
	\(K\) is a zonotope; see
	\cite[Theorem~3.5.2]{Schneider}.
\end{proof}

We can now prove the main  result of this section.

\begin{theorem}
	\label{thm:section-lifting}
	Let \(K\subset\mathbb R^n\), \(n\geq4\), be an origin-symmetric
	full-dimensional polytope, and let
	\[
	-n+4<p\leq1.
	\]
	Then the following conditions are equivalent:
	\begin{enumerate}
		\item
		\(
		K\cap H\in\mathcal I_p(H)
		\)
		for every linear hyperplane \(H\subset\mathbb R^n\);
		
		\item
		\(
		K\in\mathcal I_p;
		\)
		
		\item
		\(
		K\in\mathcal I_1;
		\)
		
		\item \(K^\circ\) is a zonotope.
	\end{enumerate}
\end{theorem}

\begin{proof}
	The equivalence of (2), (3), and (4)  was shown above. As was mentioned in the beginning of Section \ref{sec:section-lifting}, (2) implies (1).

	Finally, assume that (1) holds. Fix a hyperplane
	\(H\subset\mathbb R^n\). Since \(\dim H=n-1\), the condition on \(p\)
	can be written as
	\[
	p>-n+4=-(n-1)+3.
	\]
Thus, Theorem~\ref{thm:main-intro}, applied in the Euclidean
	space \(H\), gives
	\[
	K\cap H\in\mathcal I_1(H).
	\]
 Consequently,
	\[
	(K\cap H)^{\circ_H}
	\]
	is a zonoid. Since it is also a polytope, it is a zonotope. Using
	Lemma~\ref{lem:section-polar-projection}, we obtain
	\[
	P_HK^\circ=(K\cap H)^{\circ_H},
	\]
	so every hyperplane projection of \(K^\circ\) is a zonotope.
	Lemma~\ref{lem:hyperplane-projections-zonotope} now implies
	that \(K^\circ\) is a zonotope. Thus, (1) implies (4), and the proof is complete.
\end{proof}

\noindent {\bf AI disclosure.} 
Lemma~\ref{lem:no-jumps} was obtained with the assistance of ChatGPT. ChatGPT was also used to proofread the manuscript and to improve its
exposition and wording.

\end{document}